\documentclass[12pt]{amsart}

\usepackage{fullpage,url,amssymb,enumerate,colonequals}
\usepackage{mathrsfs}
\usepackage{MnSymbol}
\usepackage{extarrows}
\usepackage{lscape}
\usepackage[all,cmtip]{xy}
\usepackage[OT1]{fontenc}
\usepackage{color}
\usepackage[
        colorlinks, citecolor=darkgreen,
        backref,
        pdfauthor={Alex J. Best, Sander R. Dahmen, and Nuno Freitas},
]{hyperref}
\usepackage{array, booktabs, comment, multirow, todonotes}
\usepackage{todonotes}

\newcommand{\Magma}{{\tt Magma}}

\newcommand{\F}{\mathbb{F}}

\newcommand{\Q}{\mathbb{Q}}
\newcommand{\Z}{\mathbb{Z}}

\newcommand{\rhobar}{{\overline{\rho}}}
\newcommand{\calO}{\mathcal{O}}
\newcommand{\rr}{\rho}
\newcommand{\rank}{\operatorname{rank}}

\newcommand{\ff}{\mathfrak{f}}
\newcommand{\fp}{\mathfrak{p}}

\newcommand{\fq}{\mathfrak{q}}

\DeclareMathOperator{\Gal}{Gal}

\DeclareMathOperator{\Norm}{Norm}

\DeclareMathOperator{\Jac}{Jac}
\DeclareMathOperator{\un}{un}

\newcommand{\vv}{\operatorname{ord}}

\numberwithin{equation}{section}

\newtheorem{theorem}{Theorem}[section]
\newtheorem{lemma}[theorem]{Lemma}
\newtheorem{corollary}[theorem]{Corollary}
\newtheorem{proposition}[theorem]{Proposition}

\theoremstyle{definition}

\theoremstyle{remark}
\newtheorem{remark}[theorem]{Remark}

\definecolor{darkgreen}{rgb}{0,0.5,0}

\begin{document}

\title{On the generalized Fermat equation $x^{13} + y^{13} = z^n$}

\author{Alex J. Best}
\address{Alex J. Best\\
UK}
\email{alex.j.best@gmail.com}

\author{Sander R. Dahmen}
\address{Sander R. Dahmen\\
Department of Mathematics\\
VU Amsterdam\\
De Boelelaan 1111\\
1081 HV Amsterdam\\
The Netherlands}
\email{s.r.dahmen@vu.nl}

\author{Nuno Freitas}
\address{Nuno Freitas\\
Instituto de Ciencias Matem\'aticas, CSIC\\
Calle Nicol\'as Cabrera\\
13--15, 28049 Madrid, Spain}
\email{nuno.freitas@icmat.es}

\date{\today}
\thanks{The first two named authors were supported by the Dutch Research Council (NWO) under Vidi grant no. 639.032.613 with grant ID \url{https://doi.org/10.61686/ZINNY85721}.\\
We thank the Max Planck Institute for Mathematics in Bonn for its hospitality on several occasions, which enabled the last two named authors to collaborate on this paper}

\begin{abstract}
Let $n \in \Z_{\geq 2}$. We study the generalized Fermat equation
\[x^{13}+y^{13}=z^n, \quad x,y,z \in \Z, \quad \gcd(x,y,z)=1.\]
Using a combination of techniques, including modular methods, classical descent, unit sieves, and Chabauty and Mordell--Weil sieve methods over number fields, we show that for $n=5$ all its solutions $(a,b,c)$ are trivial, i.e. satisfy $abc=0$.
Under the assumption of GRH, we also show that for $n=7$ there are only trivial solutions.
Furthermore, we provide partial results towards solving the equation for general $n \in \Z_{\geq 2}$, in particular that any solution $(a,b,c)$ with $13\mid c$ is trivial.
\end{abstract}

\maketitle

\section{Introduction}

Let $p,q,r \in \Z_{\geq 2}$ and consider the generalized Fermat equation
\begin{equation}\label{eqn:GFE}
x^p+y^q=z^r, \quad x,y,z \in \Z ,\quad \gcd(x,y,z)=1.
\end{equation}
See e.g.~\cite{BCDY2015} for an overview, or~\cite{RatcliffeGrechuk2015} for a list of solved cases.
A major open problem is resolving the conjecture that for all integers $p,q,r \geq 3$, the only solutions $(a,b,c)$ to~\eqref{eqn:GFE} are the trivial ones, i.e. those satisfying $abc=0$
(for either a proof or a refutations, a one million US\$ prize is oﬀered through the American Mathematical Society by Andy Beal~\cite{Mauldin1997}).
However, solving such generalized Fermat equations (with unit coefficients) seems to be grinding to a halt.

In this work, we illustrate, by combining and strengthening a variety of modern techniques, how to overcome certain difficulties and thereby push the boundary of completely resolving generalized Fermat equations~\eqref{eqn:GFE} of central interest.
Moreover, several aspects of our approach may be of independent interest, most notably the extended applicability of a Frey curve in Section~\ref{sec:multiFrey} and the highly effective unit sieves in Section~\ref{sec:unit-sieves};
potential applications include resolving generalized Fermat equations beyond those explicitly covered below.

We focus in this work on~\eqref{eqn:GFE} with exponent triple $(13,13,p)$, i.e.
\begin{equation}\label{GFE:13-13-p}
 x^{13} + y^{13} = z^p, \quad x,y,z \in \Z ,\quad \gcd(x,y,z)=1
 \end{equation}
(where still $p \in \Z_{\geq 2}$).
It suffices to restrict to prime exponent $p$.
The equation has been solved for $p=2$ in~\cite[Theorem 1.1]{BennetSkinner}, $p=3$ in~\cite[Theorem 1.5]{BenVatYaz}, and $p=13$ as a special case of Fermat's Last Theorem (in fact, Kummer's classical 19-th century work suffices for this, as $13$ is a regular prime), and prior to this work for no other prime exponent $p$.
For $p=2,3,13$ there are only trivial solutions.
Our main focus will be on the cases $p=5$ and $p=7$, though several partial results will be given for other prime exponents $p$ as well. We will also treat the equation for all $p$ under a natural divisibility condition.
Our main result is as follows.
\begin{theorem}
 \label{thm:13-13-5or7}
 Let $p \in \{5,7\}$ and assume GRH if $p=7$.
Then the only solutions to the generalized Fermat equation
\[x^{13} + y^{13} = z^p, \quad x,y,z \in \Z ,\quad \gcd(x,y,z)=1\]
are the trivial solutions $(\pm 1,\mp 1,0), (\pm 1, 0, \pm 1)$, and $(0, \pm 1, \pm 1)$.
\end{theorem}

It is natural to distinguish between two cases of solutions $(a,b,c)$ to~\eqref{GFE:13-13-p}, namely $13 \nmid c$ and $13 \mid c$.
The latter case can be handled for any exponent, meaning that we will obtain the following result.
\begin{theorem}\label{thm:1313n}
For all integers~$n \geq 2$, the generalized Fermat equation
\[x^{13} + y^{13} = z^n, \quad x,y,z \in \Z ,\quad \gcd(x,y,z)=1\]
has no non-trivial solutions $(a,b,c)$ such that $13 \mid c$ (or equivalently $13 \mid a+b$).
\end{theorem}

The outline of this paper is as follows.
In Section~\ref{sec:HE-curves}, we reduce the resolution of~\eqref{GFE:13-13-p} in case $13\nmid c$ for a fixed odd prime $p$ to determining rational points on multiple hyperelliptic curves of genus $(p-1)/2$ over a cubic number field, after which we give a more detailed overview of proof ingredients for Theorem~\ref{thm:13-13-5or7} (and beyond).
In Section~\ref{sec:multiFrey}, we prove Theorem~\ref{thm:1313n} by Hilbert modular methods with a multi-Frey approach. In doing so, we extend the general applicability of a Frey curve studied in~\cite{BCDF}, which could be of independent interest in other contexts.
We introduce strong \lq unit sieves\rq\ in Section~\ref{sec:unit-sieves}, to reduce the number of hyperelliptic curves (per prime $p$) from Section~\ref{sec:HE-curves} that need to be considered.
After preliminary work in Section~\ref{subsec:factor-extraneous-unit}, we reduce for a range of primes $p$, including $5$ and $7$, to two hyperelliptic curves in Section~\ref{subsec:non-modular-sieve} without using modular methods, after which in Section~\ref{subsec:modular-sieve} we reduce further to one curve by also invoking modular methods. These strong sieving techniques could also be employed in other settings.
For $p=5$ and, for $p=7$ assuming GRH, we use Chabauty methods in Section~\ref{sec:Cp-points} to determine all rational points on the final hyperelliptic curve left in both cases, thus completing the proof of Theorem~\ref{thm:13-13-5or7}.
Finally, we look back and discuss in Section~\ref{sec:alternative-methods} the (non-)applicability of alternative methods, illustrating the complementarity of our methods used.

\subsection*{Tool and computational resource disclosure}

Computer algebra computations in this paper were carried out with \Magma~\cite{Magma}. Code, output, and some execution metadata are publicly available in our repository~\cite{BDF-Repository}.

We did not use generative AI for the mathematical content of this paper or the associated Magma code. We also did not use generative AI to write the text, with the exception of minor copyediting (checking spelling and grammar).

\subsection*{Notation and conventions}

For a number field $F$ we denote its ring of integers by $\calO_F$.

We let $\zeta$ denote a primitive $13$-th root of unity, $L = \Q(\zeta)$, and $K$ the cubic subfield of $L$.
Explicitly, let $\rr:=\zeta+\zeta^{-1}+\zeta^5+\zeta^{-5}$, then $K=\Q(\rr)$ and $\rr^3 + \rr^2 - 4\rr + 1 = 0$.
Note that $K$ is totally real, with a fundamental system of units for its ring of integers given by $\rr$ and $1-\rr$.

\section{Reduction to hyperelliptic curves}\label{sec:HE-curves}

Fix an odd prime $p$. In this section, we will reduce the resolution of our main equation of interest~\eqref{GFE:13-13-p} to determining points on finitely many hyperelliptic curves of genus $(p-1)/2$ over the cubic number field $K$.
After this, we will give an overview of how this fits together with result from later sections to obtain our main results mentioned in the introduction.

The polynomial $x^{13}+y^{13}\in \Z[x,y]$ factorizes over $\Q$ as
\[  x^{13} + y^{13} = (x+y)\phi_{13}\]
where~$\phi_{13} =  \frac{x^{13} + y^{13}}{x + y} \in \Z[x,y]$ is the two variable 13-th cyclotomic polynomial.
Over $K$ we get a further factorization of the form
\begin{equation}\label{eqn:K-factorization}
x^{13} + y^{13} = F \cdot \sigma(F) \cdot \sigma^2(F) \cdot (x+y)
\end{equation}
where $F \in \calO_{K}[x,y]$ is homogeneous of degree 4 and $\sigma$ is a generator for $\Gal(K/\Q)$.
Explicitly, we will choose
\[F := x^4 + \rr x^3 y + (\rr^2 + \rr - 1) x^2 y^2 + \rr x y^3 + y^4.\]
As this binary form is symmetric, one readily finds convenient identities for it. Namely, define binary forms
\[G:=x+y, \quad H:=x^2 + \frac{1}{5}(-2 \rr^2 + 8) x y + y^2\]
and constants
\[
\pi_{13}:=F(1,-1)=\rho^2-\rho+1, \quad d:=\frac{1}{4 \rr^2}=\left( \frac{\rr^2  + \rr -4}{2} \right)^2,
\]
noting that $\pi_{13}$ is a generator for the unique prime ideal in $\calO_{K}$ lying above $13$
and that
\[1+d=\pi_{13} \left(\frac{\rr^2+\rr-5}{2}\right)^2.\]
Then we have the identity
\begin{equation}\label{eqn:FGH}
(1+d) H^2 = F+d G^4.
\end{equation}

Now let $(a,b,c) \in \Z^3$ be a solution to~\eqref{GFE:13-13-p}. Since $a, b$ are coprime, we have (see e.g.~\cite[Lemma 2.2]{DahmenSiksek}) the elementary properties that
\begin{equation}\label{eqn:gcd-1or13}
\gcd(a+b,\phi_{13}(a,b))  \in \{1,13\}
\end{equation}
and
\begin{equation}\label{eqn:13-adic_info_Zfactors}
\gcd(a+b,\phi_{13}(a,b))=13 \Leftrightarrow 13 \mid c \Leftrightarrow 13 \mid a+b \Leftrightarrow 13 \mid \phi_{13}(a,b) \Leftrightarrow 13\Vert \phi_{13}(a,b).
\end{equation}
Furthermore, $F(a,b), \sigma(F(a,b)), \sigma^2(F(a,b))$ (and of course $G(a,b)=a+b$) are pairwise coprime away from the (unique) prime in $K$ above $13$.

From now on, assume that the solution satisfies $13\nmid c$, which is equivalent to $13\nmid a+b$.
(The corresponding hyperelliptic curves approach in case $13 \mid a+b$ is briefly discussed in Section~\ref{subsec:alternative-Chabauty}, while Section~\ref{sec:multiFrey} completely covers that case using modular methods.)
From the factorization~\eqref{eqn:K-factorization} and the above mentioned coprimality, we get by classical descent
\begin{equation}\label{eqn:caseI-descent-K}
    F(a,b)=e z_1^p \text{ and } G(a,b)=z_2^p
\end{equation}
for certain $e, z_1, z_2 \in \calO_{K}$ with $e$ a unit (and actually $z_2 \in \Z$ and $z_1,z_2$ nonzero).
Writing
\[Y':=\frac{H(a,b)}{z_2^{2p}}, \qquad X':=\frac{z_1}{z_2^4},\]
we see that specializing~\eqref{eqn:FGH} at $(a,b)$ and dividing by $z_2^{4p}$, we arrive at
\begin{equation}\label{eqn:Cprime-pe}
C'_{p,e}: \quad (1+d)Y'^2=e X'^p+d.
\end{equation}
This defines a hyperelliptic curve $C_{p,e}'$ of genus $(p-1)/2$.
Via rescaling $(X',Y') \mapsto (X,Y):=(X_0 X', Y_0 Y')$ where
\begin{equation}\label{eqn:X0Y0}
X_0:=4 \pi_{13}, \quad Y_0:=2^{p-1} \pi_{13}^{(p+1)/2} (\rr^2+\rr-5),
\end{equation}
it is isomorphic to the curve given by the $\calO_{K}$-integral model
\begin{equation}\label{eqn:Cpe}
C_{p,e}: \quad Y^2=e X^p+4^{p-1} \pi_{13}^p \rr^{-2}.
\end{equation}
For later reference, we note (through equality with $(X')^p/(Y')^2$) the relation 
\begin{equation}\label{eqn:GFE-sols-from-points}
\frac{(X/X_0)^p}{(Y/Y_0)^2}=\frac{F(a,b)}{eH(a,b)^2}.
\end{equation}

Determining $C_{p,e}(K)$ for all units $e \in  \calO_{K}^*$ up to $p$-th powers solves~\eqref{GFE:13-13-p} for the case $13\nmid a+b$.
Note (for any $p, e$) that $C_{p,e}(K)$ is never empty, as it contains the point at infinity, denoted~$\infty$. We summarize.

\begin{lemma}\label{lem:CaseI-solutions-rational-points}
Let $S \subset \calO_K^*$ be a complete set of representatives of $\calO_K^*/(\calO_K^*)^p$.
Then for any solution~$(a,b,c)$ to~\eqref{GFE:13-13-p} with $13\nmid c$, there is some unit $e \in S$ and a finite (i.e. affine) point $(X,Y) \in C_{p,e}(K)$ satisfying~\eqref{eqn:GFE-sols-from-points}.
\end{lemma}

The case $e=1$ is of special interest to us, and we write $C_p:=C_{p,1}$ and $C'_p:=C'_{p,1}$.
We see from~\eqref{eqn:Cprime-pe} immediately that also $(1,\pm 1)$ is contained in $C'_p(K)$.
So that 
\begin{equation}\label{eqn:3-Cp-points}
\{ (X_0,\pm Y_0), \infty \} \subset C_p(K).
\end{equation}
We note that the  trivial solutions $(a,b,c)$ with $ab=0$ (i.e. $(\pm1, 0, \pm 1)$ and $(0, \pm 1, \pm 1)$) give rise to the point $(X_0,Y_0)$.
Conversely, any potential solution $(a,b,c)$ to~\eqref{GFE:13-13-p} with $13 \nmid c$ that gives rise to one of the two finite rational points $(X_0,\pm Y_0)$ in~\eqref{eqn:3-Cp-points}, can readily be checked using~\eqref{eqn:GFE-sols-from-points} to be a trivial solution with $ab=0$ (as $F-H^2$ factors as $xy$ times a binary quadratic form irreducible over $K$).

\subsection*{Overview of proof ingredients}

For primes $p\geq 5$ with $p\ne 13$ and $p \not\equiv 1 \pmod{13}$, the possible units $e$ can be restricted by unit sieves, as explained in Section~\ref{sec:unit-sieves} below.
A \lq surviving\rq\ unit $\epsilon \in \calO_L^*$ by the sieve gives rise to a unit $e:=\Norm_{L/K}(\epsilon)\in \calO_K^*$ here.
We employ a first sieve (not relying on the modular method), introduced in Section~\ref{subsec:non-modular-sieve}, for all primes $5 \leq p \leq 47$, $p\not=13$.
It turns out that for every such prime $p$, it remains to consider only two units up to $p$-th powers, i.e. the set $S$ in Lemma~\ref{lem:CaseI-solutions-rational-points} can be reduced to size $2$ for these primes.
These units (up to $p$-th powers) include $e=1$ due to trivial solutions, and another \lq extraneous\rq\ unit; see~\eqref{eqn:norm_ext_unit} for the latter.
Explicitly, for $p=5$, it suffices to consider the two units
\[ e \in \left\{ 1, \left( \rr (1-\rr) \right)^{-1} \right\}.\]
And for $p=7$, it suffices to consider the two units
\[ e \in \left\{ 1, \left( \rr (1-\rr) \right)^3 \right\}.\]
A further unit sieve (using the modular method), given in Section~\ref{subsec:modular-sieve}, will eliminate the extraneous unit for $p=5,7$, as well as for several larger primes $p$.
Explicitly, for a prime range, it simplifies Lemma~\ref{lem:CaseI-solutions-rational-points} as follows.
\begin{theorem}\label{thm:CaeseI-13-13-p}
Fix a prime $5 \leq p \leq 47, p\neq 13$.
Let $(a,b,c)$ be a solution to~\eqref{GFE:13-13-p} with $13\nmid c$.
Then there is a finite (i.e. affine) point $(X,Y) \in C_p(K)$ satisfying~\eqref{eqn:GFE-sols-from-points}.
\end{theorem}
Here we recall
\begin{equation}\label{eqn:Cp}
C_p: \quad Y^2=X^p+4^{p-1} (\rr^2-\rr+1)^p \rr^{-2}.
\end{equation}

Searching for $K$-rational points on $C_5$ (which is quicker via $C_5'$) reveals $2$ more points beyond $(X_0,\pm Y_0), \infty$, namely the pair $(X_1, \pm Y_1)$ defined via
\begin{equation}\label{eqn:X1Y1}
X_1:=4 \rr - 4, \quad Y_1:= 176 \rr^2 - 288 \rr + 96.
\end{equation}
A similar search for $K$-rational points on $C_p$ for primes $5<p\leq 47$ only reveals the $3$ obvious points $(X_0,\pm Y_0), \infty$.

The above is about solutions with $13 \nmid c$. In the next section, we will completely deal with solutions in the case $13 \mid c$, i.e. obtain (independently of this section) Theorem~\ref{thm:1313n}.
Combining both, yields a clean (and plausible) condition for nonexistence of all possible non-trivial solutions to our main equation~\eqref{GFE:13-13-p}.

\begin{corollary}\label{Cor:GFE-via-rational-points}
Fix a prime $5 \leq p \leq 47, p\neq 13$.
If 
\begin{equation}\label{eqn:rational-points-condition}
C_p(K) =
\begin{cases}
  \{(X_0,\pm Y_0), (X_1,\pm Y_1), \infty \} & \text{for } p = 5; \\
  \{(X_0,\pm Y_0), \infty \}  & \text{for } p > 5,
\end{cases}
\end{equation}
then the trivial solutions are the only solutions to~\eqref{GFE:13-13-p}.
\end{corollary}

\begin{proof}
By Theorem~\ref{thm:1313n}, there are no nontrivial solutions with $13|c$.
Now by Theorem~\ref{thm:CaeseI-13-13-p}, we only have to show that any finite point in $C_p(K)$ does not correspond, via \eqref{eqn:GFE-sols-from-points}, to a nontrivial solution.
As noted before, the points $(X_0,\pm Y_0)$ only lead to the trivial solutions $(a,b,c)$ with $ab=0$ (as $F-H^2$ only has the linear factors $x$ and $y$ over $K$).
This establishes the corollary in case $p>5$.

For $p=5$, it remains to check that the two \lq extra\rq\ points $(X_1,\pm Y_1)$ on $C_5$ do not correspond to solutions of~\eqref{GFE:13-13-p}.
By~\eqref{eqn:GFE-sols-from-points}, it suffices to show that the binary quartic form
\[F(x,y)-\frac{(X_1/X_0)^p}{(Y_1/Y_0)^2} H(x,y)^2\]
does not have a linear factor over $K$ (or even $\Q$ actually).
One readily checks that this is indeed the case by factorizing the form over $K$ using \Magma.
(It factorizes into two irreducible quadratic forms over $K$, noting that a priori it is a difference of two squares over $K$.)
\end{proof}

We will discuss $K$-rational points on $C_p$ for $p=5,7,11$ in Section~\ref{sec:Cp-points}.
In particular, we will show that~\eqref{eqn:rational-points-condition} holds for $p=5$ and, assuming GRH, for $p=7$; see Propositions~\ref{prop:C5K} and~\ref{prop:C7K} respectively.
By Corollary~\ref{Cor:GFE-via-rational-points}, this then finished the proof of Theorem~\ref{thm:13-13-5or7}.

\section{A multi-Frey approach to Theorem~\ref{thm:1313n}}\label{sec:multiFrey}

In this section, we will present a proof of Theorem~\ref{thm:1313n} via the modular method using two Frey elliptic curves. Along the way, we will obtain some other results of independent interest.

In \cite{BCDF}, the authors study in great detail two Frey curves
(originally introduced in~\cite{DF2, Frrp}) associated with the Fermat-type equations of the form
\begin{equation}\label{eq:1313dp}
x^{13} + y^{13} = dz^p.
\end{equation}
More precisely, one Frey curve~$E_{a,b}$ is defined over $\Q(\sqrt{13})$ and the other~$F_{a,b}$ is defined over the totally real cubic subfield~$K$ of~$L=\Q(\zeta)$;
see~\cite[\S 7]{BCDF} for definitions and various properties. The exposition in {\it loc.~cit.} is oriented towards the case $d=3$, but most of the constructions and results there apply for~$d=1$, including the definitions of the Frey curves, as these are built from factors of the left hand side of~\eqref{eq:1313dp}. Moreover, in {\it loc.~cit.} there is the underlying condition that all primes~$\ell \mid d$ satisfy $\ell \not\equiv 1 \pmod{13}$, which is clearly true for $d=1$.

Let~$\fq_{13}$ be the unique prime in~$K$ above~$13$ (so it equals $\pi_{13}\calO_K$) and note that $2$ is inert in both~$\Q(\sqrt{13})$ and~$K$.
We have devised two proofs for Theorem~\ref{thm:1313n}. The more direct proof follows an application of the modular method using only the Frey curve $F_{a,b}/K$. This requires elimination of Hilbert newforms at levels $2\fq_{13}$ and~$2^3\fq_{13}$, which makes the proof computationally heavy.
The less direct proof, which we will give below, uses the multi-Frey technique combining both $E_{a,b}$ and $F_{a,b}$. The reasons for opting to present the less direct proof are the following:

(i) Some of the results in \cite[\S 7.1]{BCDF} do not apply directly to our main case of interest, i.e. $d = 1$ (and $13|c$, equivalently $13|a+b$) in~\eqref{eq:1313dp}, because they rely on the additional assumption $3 \mid a+b$ or $3 \mid d$, which we do not have. Thus our proof requires establishing some properties of $E_{a,b}$ with arguments independent of~$d$, expanding the usefulness of this Frey curve.
Some of these properties (for $d=1$) will also be applied in the unit sieve in Section~\ref{subsec:modular-sieve} already.

(ii) In the part of the argument that uses $F_{a,b}$, we then only need to do elimination of newforms at level~$2\fq_{13}$, reducing significantly the computational time of the proof. Furthermore, the more interesting parts of the elimination step in the more direct proof occurs at the level~$2\fq_{13}$ and so it is also present in the proof we give below.

\subsection*{Initial proof ingredients}

It suffices to prove Theorem~\ref{thm:1313n} for all $n = p$ a prime number.
As pointed out in the introduction, for $p=2,3,13$ the equation~\eqref{GFE:13-13-p} has already been resolved completely and has indeed only trivial solutions.
So it suffices to deal with $n=p \geq 5$ prime and $p \neq 13$ for the rest of this section.

Suppose that $(a,b,c)$ is a solution to~\eqref{GFE:13-13-p} with prime exponent~$p$.
Let $E = E_{a,b}$ over $\Q(\sqrt{13})$ be the Frey elliptic curve attached to~$(a,b,c)$
as defined in~\cite[p. 8666]{BCDF}. We denote by~$\rho_{E,p}$ the
$p$-adic Galois representation attached to~$E$ and by $\rhobar_{E,p}$ the
mod~$p$ reduction of $\rho_{E,p}$, i.e., the $p$-torsion Galois representations attached to~$E$.

\begin{proposition} Let~$p \ge 5$ prime, $p \not= 13$. Then the representation $\rhobar_{E,p}$ is irreducible.
\label{prop:irredE}
\end{proposition}
\begin{proof} Set $M = \Q(\sqrt{13})$ and recall that 2 is inert in~$M$. Let $M_2$ be the completion of~$M$ at~$2$ and $M_{2}^{\un}$ its maximal unramified extension.
Let also~$I_2 \subset G_M$ be an inertia subgroup at~$2$.

For $p \geq 7$ the result follows directly from \cite[Proposition~8]{BCDF}, so we are left with $p=5$.

Suppose that $\rhobar_{E,5}$ is reducible, that is,
\[ \rhobar_{E,5} \simeq \begin{pmatrix} \theta & \star\\ 0 & \theta' \end{pmatrix}
\quad \text{with} \quad \theta, \theta' : G_M \rightarrow \F_5^*
\quad \text{satisfying} \quad \theta \theta' = \chi_5\]
where $\chi_5$ is the mod~$5$ cyclotomic character. Thus the order of
$\rhobar_{E,5}(I_2)$ divides~$80 = 2^4 \cdot 5$.

From the proof of \cite[Proposition~3.3]{DF2}
we know that~$E$ has potentially good reduction at~$2$ and
$\vv_2(\Delta_m) = 4$ where~$\Delta_m$ is the discriminant of a minimal model for~$E$. Let $N/M_{2}^{un}$ be the extension of minimal degree over which $E/M_2$ acquires good reduction.
Denote by $e(E)$ the degree of~$N/M_{2}^{\un}$.
By Neron--Ogg--Shafarevich we know that $\rho_{E,5}(I_2) \simeq \Gal(N/M_{2}^{un})$
has order $e(E)$.

Since $\vv_2(\Delta_m) \not\equiv 0 \pmod{3}$ it follows from \cite[Th\'eor\`em 3]{Kraus1990} that $3 \mid e(E)$ and $e(E) \mid 24$. Moreover, since $5 \nmid e(E)$, reduction modulo~$5$
preserves the order of $\rho_{E,5}(I_2)$, thus $\rhobar_{E,5}(I_2)$ also has order $e(E)$, giving a contradiction because $3 \nmid 80$.
\end{proof}

\begin{proposition}
Let~$p \ge 5$ prime, $p \not= 13$. Then
\begin{equation}
\rhobar_{E,p} \simeq \rhobar_{Z,p}
\label{E:trivialisos}
\end{equation}
where the elliptic curve $Z$ equals $E_{1,-1}$, $E_{1,0}$, or $E_{1,1}$.
\label{P:iso2}
\end{proposition}
\begin{proof} For $p =11$ and $p \geq 17$ this follows directly from~\cite[Proposition 9]{BCDF}. For $p=7$ the same proposition includes the additional possibility that
$\rhobar_{E,p} \simeq \rhobar_{g,\fp_7}$ for a Hilbert newform $g$ over $\Q(\sqrt{13})$ of parallel weight $2$, trivial character, level $2^3 \cdot 13$, with field of coefficients $\Q(\sqrt{2})$, and a choice of prime $\fp_7$ above $7$ in this field. However, as explained in Remark 7.4 of {\it loc.~cit.} and proved in~\cite[Proposition~6.1]{BCDDF} we have
$\rhobar_{g,\fp_7} \simeq \rhobar_{E_{1,-1},7}$, which is already among the cases in the statement, completing the proof for $p=7$ as well.

Note that the conclusion of \cite[Proposition 9]{BCDF} also holds for $p=5$
under the additional hypothesis $3 \mid a+b$, which we do not have. This hypothesis is there to guarantee irreducibility of~$\rhobar_{E,5}$ via \cite[Proposition~8]{BCDF} and consequently apply level lowering \cite[Lemma~7]{BCDF}. In our setting, we have irreducibility of $\rhobar_{E,5}$ by Proposition~\ref{prop:irredE} and everything else in {\it loc.~cit.} applies exactly the same, yielding the result for $p=5$.
\end{proof}

We remark that all the results above did not use the assumption $13 \mid a+b$ in Theorem~\ref{thm:1313n}.

We now come to our main application of the Frey curve $E$. The following arithmetic information will be key in restricting the possible levels of Hilbert newforms when using the other Frey curve $F$.

\begin{theorem}\label{thm:overQ13}
Let $(a,b,c)$ be a solution to~\eqref{GFE:13-13-p} with prime exponent $p \geq 5$, $p \neq 13$.
If $13 \mid a+b$, then $2 \mid a+b$.
\end{theorem}

\begin{proof}
Let $E = E_{a,b}$ be the Frey curve associated with $(a,b,c)$.
From Proposition~\ref{P:iso2}, we know that $\rhobar_{E,p} \simeq \rhobar_{Z,p}$, where $Z$ is
$E_{1,-1}$, $E_{1,0}$, or $E_{1,1}$.

Let $K^+$ be the maximal totally real subfield of $L=\Q(\zeta)$ and $\pi$ denote the unique prime ideal in~$K^+$ above 13. From \cite[Proposition~3.1]{DF2}, the base change curves
$E_{1,0}/K^+$ and $E_{1,1}/K^+$ have bad additive reduction at~$\pi$, while~$E_{1,-1}/K^+$ has good reduction at~$\pi$. Furthermore, $E_{a,b}/K^+$ also has good reduction at~$\pi$ because $13 \mid a+b$. Therefore, for $Z = E_{1,0}$ and $Z=E_{1,1}$,
restricting the isomorphism $\rhobar_{E,p} \simeq \rhobar_{Z,p}$ to~$G_{K^+}$ gives a contradiction because
$\rhobar_{E,p}|_{G_{K^+}}$ is unramified at~$\pi$ whilst $\rhobar_{Z,p}|_{G_{K^+}}$ ramifies at~$\pi$. We conclude that $\rhobar_{E,p} \simeq \rhobar_{E_{1,-1},p}$.
Now the proof of part (B) in \cite[Theorem 7]{BCDF} applies exactly the same to conclude $4 \mid a+b$. So trivially, $2 \mid a+b$.

(The argument in the proof of \cite[Theorem 7]{BCDF} part (B) is purely local at~$2$ and so independent of the condition $3 \mid d$ in {\it loc.~cit.}; indeed, this condition is used there only to guarantee $\rhobar_{E,p} \simeq \rhobar_{E_{1,-1},p}$, which we established independently of~$d$.)
\end{proof}

\subsection*{Completing the proof of Theorem~\ref{thm:1313n}}

We will now work with the Frey elliptic curve $F = F_{a,b}$ over $K$ as defined in~\cite[p. 8669]{BCDF}. The results in {\it loc.~cit.} regarding~$F_{a,b}$ are stated for general~$d$ (in particular, independently of $3 \mid d$) and apply directly in our setting, i.e. $(a,b,c)$ is a solution to~\eqref{GFE:13-13-p} with prime exponent~$p \geq 5$, $p \neq 13$ and $c\neq 0$ (which guaranties that $F$ is nonsingular).
Furthermore, we assume that $13 \mid a+b$. So by Theorem~\ref{thm:overQ13}, we have $2 \mid a+b$. As $a+b$ and $\phi_{13}(a,b)=(a^{13}+b^{13})/(a+b)$ are coprime away from $13$ (see~\eqref{eqn:gcd-1or13}), this gives $2^p \mid a+b$. 

We now apply modularity, irreducibility, and level lowering results for the Frey curve $F$.
Concretely, from Lemmas~8--11 and Theorem~8 of~\cite{BCDF}, it follows that $\rhobar_{F,p}$ is irreducible and
\begin{equation*}
  \rhobar_{F,p} \simeq \rhobar_{f,\fp}
  \label{E:Fiso}
\end{equation*}
where $f$ is a Hilbert newform over $K$ of parallel weight $2$, trivial character, level~$N=2 \fq_{13}$, and $\fp$ a prime above~$p$ in the field of coefficients of~$f$. 
We note that in the computation of the level $N$, we used the info $\operatorname{ord}_2(a+b)\geq 3$ to yield $\operatorname{ord}_2(N)=1$ (cf.~\cite[Lemma 9]{BCDF}).
Furthermore, we note that level lowering also applies at $3$ (the given model is not minimal at $3$, but $F$ is still semistable at $3$ with 3-adic valuation of the minimal discriminant divisible by $p$ etc.), so the case $3|N$ in \cite[Lemma 11]{BCDF} can be ignored.

We compute the above space using \Magma.
There are four newforms in the space, say $f_1, f_2, f_3$, and $f_4$. The forms~$f_1$, $f_2$ have rational coefficients and the forms $f_3$, $f_4$ have cubic coefficients fields. Furthermore, the form~$f_3$ is the form denoted by~$\ff_{11}$ in~\cite[\S 8]{BCDDF} and its field of coefficients $\Q_{f_3}$ is the maximal totally real subfield of~$\Q(\zeta_7)$.
Let $\fp_7$ denote the unique prime in $\Q_{f_3}$ above~$7$.

Using \Magma for standard trace comparisons at the auxiliary primes $q = 5,7,11$, eliminates the four newforms for all exponents~$p$ except for $f_1$ and $f_3$ when $p=7$. In other words, it could still be possible that
$\rhobar_{F,7} \simeq \rhobar_{f_1,7}$ or $\rhobar_{F,7} \simeq \rhobar_{\ff_{11},\fp_7}$.
We claim that both $\rhobar_{f_1,7}$ and $\rhobar_{\ff_{11},\fp_7}$ are reducible. Therefore, the previous isomorphisms cannot happen because $\rhobar_{F,7}$ is irreducible.

We now prove the claim. Let $W$ be the base change to~$K$ of the elliptic curve with Cremona label $26b1$. The conductor of $W$ is $2\fq_{13}$ and~$W$ is modular because~$K$ is cubic and totally real~\cite{modularityCubic}. Thus either~$f_1$ or~$f_2$ correspond to $W$ via modularity and comparing the trace of Frobenius at $3\calO_K$ shows that $f_1$ corresponds to~$W$.
This curve has a $7$-torsion point over~$K$ (in fact over~$\Q$), therefore $\rhobar_{f_1,7}$ is reducible.
Finally, the representation $\rhobar_{\ff_{11},\fp_7}$ is also reducible by~\cite[Proposition~8.3]{BCDDF}, establishing the claim.
This also finishes the proof of Theorem~\ref{thm:1313n}.

\begin{remark} We note that we actually have $\rhobar_{f_1,7} \simeq \rhobar_{\ff_{11},\fp_7}$, allowing for a variation of the proof of the claim. Indeed, the previous isomorphism follows from an application of the socle method in \cite[\S 6]{BCDDF}. Therefore, the claim follows if we show that any of the two representations is reducible, which we can do by using either of the arguments in the proof.
\end{remark}

\section{The unit sieves}\label{sec:unit-sieves}

Throughout this section, let $p$ denote a rational prime with $p\not=2,3, 13$.
In this section, we develop, and for $p \leq 47$ employ, sieves to restrict the possible units $e \in \calO_K^*$ modulo $p$-th powers that can occur in~\eqref{eqn:caseI-descent-K}, thereby restricting the number of hyperelliptic curves~\eqref{eqn:Cpe} to consider in the case $13 \nmid c$, and similarly for the case $13 \mid c$.
In particular, we will establish Theorem~\ref{thm:CaeseI-13-13-p}.
Together with Theorem~\ref{thm:1313n}, established in the previous section, this will also yield Corollary~\ref{Cor:GFE-via-rational-points}, as explained at the end of Section~\ref{sec:HE-curves}.

The arguments in this section are inspired by the sieves in~\cite[\S 4]{DahmenSiksek} and~\cite[\S 7]{BCDDF}. The key difference is that here we will work locally at $p$, which is one of the exponents in~\eqref{GFE:13-13-p}, whilst in {\it loc.~cit.} all primes used in the sieve are different from the exponents. This allows us to work modulo $p^2$, resulting in highly effective sieves.

\subsection{Factorization and extraneous unit}\label{subsec:factor-extraneous-unit}

Suppose $(a,b,c)$ is a solution to~\eqref{GFE:13-13-p}. We have the factorization in $\calO_L$,
\begin{equation}\label{eq:factorization}
  a^{13} + b^{13} = (a+b)\phi_{13}(a,b)= (a+b) \prod_{k=1}^{12} (a + \zeta^k b) = c^p.
\end{equation}

We note some elementary facts regarding~\eqref{eq:factorization}.
Let~$\fp_{13}$ be the unique prime in~$L$ above~$13$, which satisfies $\vv_{\fp_{13}}(13) = 12$ and is generated by $1-\zeta$.
Recall that since $a,b$ are coprime,
the integers \(a + b\) and~$\phi_{13}(a,b)$ are coprime away from~\(13\) (see~\eqref{eqn:gcd-1or13}) and if $13 \mid a+b$ then $13 \Vert \phi_{13}(a,b)$ (see~\eqref{eqn:13-adic_info_Zfactors}).
Furthermore (see e.g~\cite[\S 2.1]{Frrp}), the factors $a + \zeta^k b$ for $1 \leq k \leq 12$ are pairwise coprime away from~$\fp_{13}$, they satisfy $\vv_{\fp_{13}}(a + \zeta^k b) = 1$ when $13 \mid a + b$, and all primes $\ell \neq 13$ dividing $\phi_{13}(a,b)$ satisfy $\ell \equiv 1 \pmod{13}$.

Therefore, from~\eqref{eq:factorization} and classical descent, we have
\begin{equation}
\label{unit-equation}
  a + \zeta b = \begin{cases}
    \epsilon \gamma^p  & \text{ if } 13 \nmid a+b; \\
    \epsilon (1 - \zeta) \gamma^p & \text{ if } 13 \mid a+b,
   \end{cases}
\end{equation}
for some $\epsilon \in \calO_L^*$ and~$\gamma \in \calO_L$; and also (recalling $p\not=2$),
\begin{equation}
\label{pair-condition}
  a + b = \begin{cases}
    \delta^p  & \text{ if } 13 \nmid a+b; \\
    13^{pj-1}   \delta^p & \text{ if } 13 \mid a+b,
   \end{cases}
\end{equation}
for some~$\delta \in \Z$ and~$j \geq 1$.

Moreover, if $p \mid \phi_{13}(a,b)$ (and hence $p \mid a^{13} + b^{13}$) then $p \equiv 1 \pmod{13}$.
From now on assume, next to $p\not= 2,3, 13$, that $p\not\equiv 1 \pmod{13}$.
Then $p \nmid \phi_{13}(a,b)$, and consequently $\fp \nmid a+\zeta^k b$ for $1 \leq k \leq 12$ and all primes $\fp \mid p$ in~$L$.
Let $p\calO_L = \fp_1 \cdot \ldots \cdot \fp_s$ be the prime factorization of the prime $p$ in $L$.  We can reduce~\eqref{unit-equation} modulo~$\fp_i^2$, and from $\fp_i \nmid a+\zeta b$ it follows that $a+\zeta b \pmod{\fp_i^2}$ is invertible in $\calO_L/\fp_i^2$. Since the order of the unit group of
$\calO_L/\fp_i^2$ is divisible by~$p$, the condition of being a $p$-th power mod~$\fp_i^2$
is nontrivial (note that working only mod~$\fp_i$ would give a trivial condition).

By inspection, we easily spot the following solutions to~\eqref{unit-equation} with $\gamma = \pm 1$:
\begin{itemize}
\item[(i)] $(a,b) = \pm(1,0)$ with $\epsilon = 1$;
\item[(ii)] $(a,b) = \pm (0,1)$ with $\epsilon = \zeta$;
\item[(iii)] $(a,b) = \pm (1,1)$ with $\epsilon = 1+\zeta$;
\item[(iv)] $(a,b) = \pm (1,-1)$ with $\epsilon = 1$.
\end{itemize}
Note that (i), (ii), and (iii) correspond to the case $13 \nmid a+b$, and (iv) corresponds to the case $13 \mid a+b$.

Of course, replacing $(\epsilon, \gamma)$ by $(-\epsilon, -\gamma)$ in each case above will also give solutions.  
But clearly, in~\eqref{unit-equation} we only need to consider~$\epsilon$ up to $p$-th powers.
We have $pk \equiv 1 \pmod{13}$ for some $k\in \Z$ (since $p\not=13$), which implies $\zeta = (\zeta^k)^p$.
So we see that the unit~$\epsilon$ in solutions~(i) and~(ii) are the same up to $p$-th powers; also $\epsilon$ and~$-\epsilon$ are the same modulo $p$-th powers.

In the case that $13 \mid a+b$, solution~(iv) of course also satisfies~\eqref{pair-condition} (with $\delta=0$).
There seems no a priori reason, for general $p$, to expect any other unit than $\epsilon=1$ (up to $p$-th powers) from solution~(iv) to survive the unit sieve in Section~\ref{subsec:non-modular-sieve} below, where the latter consists roughly speaking of carefully considering~\eqref{unit-equation} and~\eqref{pair-condition} modulo $p^2$.

Let us turn to the case $13 \nmid a+b$.
When $2$ is not a $p$-th power mod~$p^2$, the solution~(iii) does not satisfy~\eqref{pair-condition}, so we do not expect the corresponding unit $\epsilon = 1+\zeta$ to survive the unit sieve modulo $p^2$ in Section~\ref{subsec:non-modular-sieve} below in general.
However (up to $p$-th powers, as usual), next to the unit $\epsilon=1$,
there turns out to be a less obvious solution to~\eqref{unit-equation} which will give rise to a unit~$\epsilon=\varepsilon_0$ that will survive that sieve in general (in the case $13 \nmid a+b$ under consideration).

Indeed, to illustrate this for $p=5$, we have
\begin{equation}
 \label{E:epsilon_0}
 13^2-\zeta 13^2= \varepsilon_0 \gamma_0^5 \quad \text{ where } \quad \gamma_0:=(1-\zeta)^5, \quad \varepsilon_0:=13^2/(1-\zeta)^{24}.
\end{equation}
(Later on, $\gamma_0$ and $\varepsilon_0$ will depend on $p$.)
Although the pair $(a,b) = (13^2,-13^2)$ does not satisfy $13 \nmid a+b$, this is not detectable when working modulo~25.
This shows that for $a\equiv 13^2 \pmod{5^2}$ and $b \equiv -13^2 \pmod{5^2}$, the modulo~25 sieve can never discard the unit $\varepsilon_0$; trying to sieve at additional primes will also not eliminate this unit for the same reasons.

In the general case of exponent $p\geq 5$, considering still the case of solutions with $13 \nmid a+b$, we have that apart from the unit $1$, there is always at least one other extraneous unit surviving the unit sieve in Section~\ref{subsec:non-modular-sieve} below as follows (all up to $p$-th powers of course).
Note that we have the unit 
\[
\mu:=\frac{13}{(1-\zeta)^{12}} \in \calO_L^*, 
\]
and hence for every $k \in \Z$ the useful identity
\begin{equation}\label{eqn:creating_extraneous_unit}
  \left( 1-\zeta \right) 13^k = \mu^k (1-\zeta)^{12k+1}.
\end{equation}
For any $k \in \Z$ such that 
\begin{equation}\label{eqn:kmodp}
12k\equiv -1 \pmod{p},
\end{equation}
which has a unique solution modulo $p$ (as $p\geq 5$),
we get of course that
\[\gamma_0:=(1-\zeta)^{(12k+1)/p} \in \calO_L.\]
Now the extraneous unit
\begin{equation}\label{eqn:extraneous_unit}
    \varepsilon_0:=\mu^k=\frac{13^k}{(1-\zeta)^{12 k}} \qquad \left(\text{where}\ 12k\equiv -1 \pmod{p}\right)
\end{equation}
survives the unit sieve from Section~\ref{subsec:non-modular-sieve} below, since from~\eqref{eqn:creating_extraneous_unit} we see that
\[13^k-\zeta 13^k=\varepsilon_0 \gamma_0^p.\]
For later reference, using $\Norm_{L/K}(1-\zeta)=\pi_{13}$ and $\pi_{13}^3=\left(\rr (1-\rr) \right)^2 13$, we note that
\begin{equation}\label{eqn:mu8th-power}
\Norm_{L/K}(\mu) = \Norm_{L/K}\left(\frac{13}{(1-\zeta)^{12}}\right)=\left( \frac{13}{\Norm_{L/K}(1-\zeta)^3} \right)^4=\left( \rr (1-\rr) \right)^{-8}.
\end{equation}
For $j \in \Z$ with $j\equiv -8k \pmod{p}$, the condition~\eqref{eqn:kmodp} is equivalent to $3j\equiv 2 \pmod{p}$.
For the norm (from $L$ to $K$) of the extraneous unit $\varepsilon_0=\mu^k$, up to $p$-th powers, we therefore get
\begin{equation}\label{eqn:norm_ext_unit}
\Norm_{L/K}(\varepsilon_0) =\left( \rr (1-\rr) \right)^j \pmod{\calO_{K}^{*p}}, \qquad 3j \equiv 2 \pmod{p}.
\end{equation}
In particular, for $p=5$ we can take $j=-1$, and for $p=7$ we can take $j=3$.

\subsection{The sieve without modular information}\label{subsec:non-modular-sieve}

We recall that $p\not= 2,3, 13$ and $p \not\equiv 1 \pmod{13}$.
Let $p\calO_L = \fp_1 \cdot \ldots \cdot \fp_s$ be the factorization of~$p$ in~$\calO_L$.

{\bf Case (I).} We first apply the sieve in the case $13 \nmid a + b$.

Let $\mathcal{P}$ be the set of pairs $(\alpha,\beta)$ where $0 \leq \alpha \leq \beta \leq p^2-1$ such that at most one of~$\alpha,\beta$ is a multiple of~$p$, and $\alpha+\beta$ is a $p$-th power in $\Z/p^2\Z$; we assumed $\alpha \leq \beta$ by the symmetry in $x,y$ of~\eqref{GFE:13-13-p}.

Let $M_i := (\calO_L/\fp_i^2)^*$ and
define $\epsilon_i:=\epsilon_i(\alpha,\beta) = \alpha+\beta\zeta \pmod{\fp_i^2}$ for $(\alpha,\beta) \in \mathcal{P}$. 
Since $\alpha,\beta$ are not both divisible by $p$ and $p \not\equiv 1 \pmod{13}$, the properties of $\phi_{13}$ mentioned in the paragraph following~\eqref{eq:factorization} guarantee that $\epsilon_i \in M_i$. 
Let $\bar{\epsilon}_i$ denote the canonical image of $\epsilon_i$ in $M_i / M_i^p$ and
consider
\begin{align*}
\phi \; : \; \mathcal{P} & \to \prod M_i / M_i^p \\
      (\alpha,\beta) & \mapsto (\bar{\epsilon}_1, \ldots, \bar{\epsilon}_s). \
\end{align*}
We also have the standard reduction maps $\tilde{\pi}_i : \calO_L^* \to M_i / M_i^p$, which give rise to a map

\begin{align*}
\pi \;  : \;  \calO_L^* / \calO_L^{*p} & \to \prod M_i / M_i^p \\
      [\epsilon] & \mapsto (\tilde{\pi}_1(\epsilon), \ldots , \tilde{\pi}_s(\epsilon)) \
\end{align*}
where $[\epsilon]:= \epsilon \ \operatorname{mod} \calO_L^{*p} \in \calO_L^* / \calO_L^{*p}$ and the map is well defined by the representative $\epsilon \in \calO_L^*$.
Note that for a solution $(a,b,c)$ to~\eqref{GFE:13-13-p} satisfying~\eqref{unit-equation} for some $\gamma,\epsilon$, there exists a pair $(\alpha, \beta) \in \mathcal{P}$ (congruent to $(a,b)$ or $(b,a)$ modulo $p^2$) such that
\[\phi(\alpha,\beta) = \pi([\epsilon]).\]

Using \Magma, we explicitly construct both maps~$\pi$ and $\phi$.
As $\mathcal{P}$ is rather smaller (e.g. for $p=5,7$ of orders $62, 171$ resp.) than $\calO_L^* / \calO_L^{*p}$ (e.g. for $p=5,7$ of orders $5^5=3125, 7^5=16807$ resp.), we do not compute the intersection $\pi(\calO_L^* / \calO_L^{*p}) \cap \phi(\mathcal{P})$.
Instead, we simply compute the inverse image of $\phi(\mathcal{P})$ under $\pi$.
In practice, we assert that $\# \ker(\pi) = 1$, and then for every $u \in \phi(\mathcal{P})$, we check whether there is a (necessarily unique) $\mathcal{E} \in \calO_L^* / \calO_L^{*p}$ such that $\pi(\mathcal{E})=u$, and if so, store a representative $\epsilon \in \calO_L^*$ for the class $\mathcal{E}=[\epsilon]$.
The union of those units $\epsilon$, denoted $S'$, is the output of the sieve.
Concretely, this means that (for the prime $p$ under consideration) Lemma~\ref{lem:CaseI-solutions-rational-points} still holds true upon replacing the set $S$ with the set
\[
\{\Norm_{L/K}(\epsilon) : \epsilon \in S'\}.
\]

We ran the sieve in \Magma\ for all primes $p$ in the range $5 \leq p \leq 47$, $p\not=13$.
For $p=5$, after running the sieve, we have (up to $5$-th powers) two surviving units:
\[
 \epsilon = 1 \qquad \text{ and } \qquad
 \epsilon = -2\zeta^{11} - 4\zeta^{10} - 3\zeta^9 - 4\zeta^8 - 2\zeta^7 - 4\zeta^5 - 5\zeta^4 - \zeta^3 - \zeta^2 - 5\zeta - 4.
\]
The second unit is indeed equal to $\varepsilon_0=13^2/(1-\zeta)^{24}$ up to multiplication with a fifth power of a unit (the latter unit can be calculated to be $-2\zeta^{11} - 2\zeta^{10} + \zeta^9 - 3\zeta^8 - \zeta^7 - \zeta^6 - \zeta^5 - \zeta^4 - 3\zeta^3 + \zeta^2 - 2\zeta - 2$).

Similarly, for all other primes in our range, we also have as output for our sieve that, up to $p$-th powers, there are only 2 units surviving, namely $1$ and the extraneous unit $\varepsilon_0$.
Concretely, this means that for any prime $p$ with $5 \leq p \leq 47$, $p\not=13$,
Lemma~\ref{lem:CaseI-solutions-rational-points} still holds true upon replacing the set $S$ with the set $\{1,(\rho (1-\rho))^j\}$ for any choice of $j\in \Z$ with $3j \equiv 2 \pmod{p}$.

{\bf Case (II).} Although the case $13 \mid a + b$ has already been dealt with using the modular method, it is interesting to quickly discuss how the sieve applies in this case; see also Section~\ref{sec:alternative-methods}. There are only two  differences:

(a) Due to~\eqref{pair-condition}, the set $\mathcal{P}$ is the set of pairs $(\alpha,\beta)$ where $0 \leq \alpha \leq \beta \leq p^2-1$ such that at most one of~$\alpha,\beta$ is a multiple of~$p$, and $13(\alpha+\beta)$ is a $p$-th power in $\Z/p^2\Z$.

(b) Note that $(1-\zeta)$ is invertible mod~$\fp_i$.
The map $\phi$ is defined instead using the quantities
$$\epsilon_i:= (\alpha+\beta\zeta)(1-\zeta)^{-1} \pmod{\fp_i^2}.$$

As before, with the help of \Magma, we apply the sieve to all primes $5 \leq p \leq 47$, $p \neq 13$.
In this case, for all primes in our range, the only unit surviving the sieve is $\epsilon = 1$ (up to $p$-th powers), except when $p=17$. For the latter prime, we get two units (including the trivial one of course) in~{\bf (II)}.
Presumably, the non-trivial unit can be eliminated by sieving at another prime.
However, in light of the resolution when $13 \mid a+b$ in Theorem~\ref{thm:1313n}, we will not pursue this here.

\subsection{The sieve with modular information}\label{subsec:modular-sieve}

We consider again the case $13 \nmid a+b$, and will now eliminate the extraneous unit~$\varepsilon_0$ given in~\eqref{eqn:extraneous_unit} for a range of primes $p$, including $p=5,7$.
For this, we will combine modulo~$q$ information, for some auxiliary prime $q$, from the above sieve with that of the Frey curve~$E_{a,b}/\Q(\sqrt{13})$ from Section~\ref{sec:multiFrey}.

For sake of exposition, first let $p=5$ or~$7$, our main cases of interest.
We start by extracting modulo~$q=19$ information from the sieve.
Since $p \nmid 19-1$, we let $\mathcal{P}$ be the set of pairs $(\alpha,\beta) \neq (0,0)$ with $0 \leq \alpha \leq \beta \leq 18$. Note that $19$ is inert in~L=$\Q(\zeta)$ and
let $M_1 := (\calO_L/19\calO_L)^*$.
For the mod $q=19$ situation, we consider the map
 \begin{align*}
\phi \; : \; \mathcal{P} & \to M_1 / M_1^p \\
      (\alpha,\beta) & \mapsto \overline{\alpha+\beta\zeta \pmod{19}}. \
\end{align*}
We also have, as before, the canonical map (for the newly defined $M_1$)
\[ \pi \;  : \;  \calO_L^* / \calO_L^{*p} \to M_1 / M_1^p.\] 
Using \Magma, we find that, for $p=5$ and $p=7$, the pairs in~$\mathcal{P}$ satisfying $\phi(\alpha,\beta) = \pi([\varepsilon_0])$ are 
\[
 \mathcal{L} = \left\{ (i,19-i) \mid  i \in \{1,2,\ldots,9\}  \right\}.
\]
Now let $(a,b,c)$ be a solution to~\eqref{GFE:13-13-p} with $p=5$ or $p=7$ satisfying $13 \nmid a+b$.
Assume that after descent it gives rise to the extraneous unit~$\epsilon=\varepsilon_0$ in~\eqref{unit-equation}.
Thus $(a \mod 19,b \mod 19)$ (or the swapped pair) is in~$\mathcal{L}$.

We now consider the Frey curve $E_{a,b}/\Q(\sqrt{13})$ attached to~$(a,b,c)$. From Proposition~\ref{P:iso2} it follows that $\rhobar_{E_{a,b},p}$ is isomorphic to
$\rhobar_{E_{1,0},p}$, $\rhobar_{E_{1,1},p}$, or $\rhobar_{E_{1,-1},p}$. Since $13 \nmid a+b$, from the proof of Theorem~\ref{thm:overQ13}, we have $\rhobar_{E_{a,b},p} \not\simeq \rhobar_{E_{1,-1},p}$. Hence
\begin{equation}\label{eqn:CaseI-IsomGalReps}
 \rhobar_{E_{a,b},p} \simeq \rhobar_{E_{1,0},p} \qquad \text{ or } \qquad \rhobar_{E_{a,b},p} \simeq \rhobar_{E_{1,1},p}.
\end{equation}
Note that~$19$ is inert in~$\Q(\sqrt{13})$. Since $19 \not\equiv 1 \pmod{13}$, it follows from~\cite[Lemma~5]{BCDF} that~$E_{a,b}$ has good reduction at~$19$.
By taking traces of Frobenius, computing explicitly $a_{19}(E_{1,0})=-9$ and $a_{19}(E_{1,1})=3$, we obtain from~\eqref{eqn:CaseI-IsomGalReps}, respectively,
\begin{equation}
\label{eq:auxiliaryq}    
 a_{19}(E_{a,b}) \equiv -9 \pmod{p} \quad \text{ or } \quad a_{19}(E_{a,b}) \equiv 3 \pmod{p}.
\end{equation}

On the other hand, for all $(\alpha,\beta) \in \mathcal{L}$, we compute that $a_{19}(E_{\alpha, \beta})=25$, so
\[a_{19}(E_{a,b}) \equiv 25 \pmod{p},\]
yielding a contradiction to both congruences in~\eqref{eq:auxiliaryq} for both $p=5$ and $p=7$.
This eliminates the possibility of~$\varepsilon_0$ occurring after descent for these primes, as desired.

For the prime exponents $11\leq p \leq 47$, $p\not=13$,  we apply a similar argument using an appropriate auxiliary prime~$q$.
More precisely, we use the following pairs $(p,q)$
\[
(11,23), (17,103), (19,7), (23,139), (29,233), (31,37), (37,11) , (41, 163), (43, 7), (47,281),
\]
with the improvement that, when $p \mid q-1$, equation~\eqref{pair-condition} allows us to take $\mathcal{P}$ to be
the set of pairs $(\alpha,\beta) \neq (0,0)$ with $0 \leq \alpha \leq \beta \leq q-1$ and $\alpha+\beta$ a $p$-th power in $\Z/q\Z$.
We performed all these computations in \Magma.
The upshot is that also in this prime range, the extraneous unit is eliminated.

Together with the earlier results of the sieve without modular information from Section~\ref{subsec:non-modular-sieve}, for solutions $(a,b,c)$ to~\eqref{GFE:13-13-p} with $13 \nmid c$ in the prime range $5 \leq p \leq 47$, $p \ne 13$, we have thus eliminated all units $\epsilon$, resp. $e$, that occur after descent in~\eqref{unit-equation}, resp.~\eqref{eqn:caseI-descent-K}, except for the unit $1$ (up to respective $p$-th powers).
Hence, by the reduction to hyperelliptic curves in Section~\ref{sec:HE-curves}, this completes the proof of Theorem~\ref{thm:CaeseI-13-13-p}.

\begin{remark}
We note that for all primes $p$, the auxiliary prime $q$ given above is the smallest prime for which this method works, expect for $p=7$ where instead of $q=19$ one could have taken the smaller auxiliary prime $q=11$ (which would have slightly expanded the exposition above).
In principle, the sieve could be slightly strengthened by using two (possibly distinct) auxiliary primes $q_1, q_2$ for every $p$: one for each of the two isomorphisms in~\eqref{eqn:CaseI-IsomGalReps}.
However, choosing $q_1, q_2$ minimal, $\max(q_1,q_2)$ turns out to be equal to our minimal $q$ for all primes $p$ considered in the range above (with $q=11$ for $p=7$), except that for $p=23$, where $q=139$, one could take $(q_1,q_2)=(137, 47)$.
\end{remark}

\section{Rational points on \texorpdfstring{$C_p$}{Cp}}\label{sec:Cp-points}

We recall that solutions $(a,b,c)$ to~\eqref{GFE:13-13-p} with $13 \nmid c$ and $e=1$ in~\eqref{eqn:caseI-descent-K} give rise to $K$-rational points on the curve $C_p$ in~\eqref{eqn:Cp}. This case includes the trivial solutions with $ab=0$.
In this section, we will use variants of Chabauty's method to determine $C_p(K)$ for $p=5$ and $p=7$, where we assume GRH is the latter case.
Subsequently, this will complete the proof of Theorem~\ref{thm:13-13-5or7}.
We will also explore some partial results for $p=11$ (assuming GRH again).

Consider a (smooth, projective, geometrically integral) curve $C$ of genus $g\geq 2$ over a number field $F$, which we assume (for convenience) to contain an $F$-rational point.
In its basic form, the method of Chabauty (and Coleman) considers an $F$-embedding of $C$ inside an abelian variety $J$ (usually its Jacobian) $C \hookrightarrow J$ and studies the subsets of $F$-rational points inside these spaces.
In the case where the rank of $J(F)$ is less than the dimension of $J$, the same is true for the closure of $J(F)$ in $J(F_{\fp})$, denoted $\overline{J(F)}$, where $\fp$ is a (finite) prime of good reduction for $C$ and $F_{\fp}$ denotes the $\fp$-adic completion of $F$.
Now the intersection inside $J(F_{\fp})$ of the one-dimensional $\fp$-adic variety $C(F_{\fp})$ with the codimension at least one $\fp$-adic variety $\overline{J(F)}$ will be zero-dimensional, hence (due to compactness) finite.
In practice, given generators for a finite index subgroup of $J(F)$, we can find nontrivial locally analytic $\fp$-adic functions vanishing on all of $\overline{J(F)}$ and hence all of $C(F) \subseteq C(F_{\fp}) \cap J(F)$.
Pulling these functions back to the curve then results in an effective method to find the rational points, or at least a finite set of $\fp$-adic points containing them, assuming still of course that the rank condition is met.

Many variants of these methods have been introduced, which can in some cases apply when the original rank condition is not satisfied. These generally work by considering similar set-ups coming from different constructions starting from the initial curve.
They include elliptic curve Chabauty~\cite{Bruin2003}, quadratic Chabauty~\cite{BalakrishnanDogra2018}, Selmer group Chabauty~\cite{Stoll2017}, and, most importantly for us, number field Chabauty~\cite{Siksek2013}.

In this last variant, introduced by Siksek using ideas of Wetherell, we begin with a curve $C$ defined over a number field $F$ as above with Jacobian $J$, but rather than considering $C(F) \hookrightarrow J(F)$ we consider $\operatorname{Res}_{F/\Q} C (F) \hookrightarrow \operatorname{Res}_{F/\Q} J (F)$. This has the advantage of possibly working when $\operatorname{rank}(J(F)) \le [F:\Q](g - 1)$ (rather than $\operatorname{rank}(J(F)) \le g - 1$), although it may not always be successful.

With many Chabauty-like methods, when applicable, the $l$-adic analysis suffices to give a bound on the number of rational points in each $l$-adic disk, though often this bound is not sufficient to determine the rational points exactly.
In such situation, the combination of Chabauty with some form of Mordell--Weil sieve~\cite{BruinStoll2010} is what usually suffices to effectively determine the rational points on the curve.

We will now turn to $C_p(K)$ for $p=5,7,11$.
Recall that an (affine) equation for $C_p$ is given in~\eqref{eqn:Cp}, with known (finite) $K$-rational points $(X_0,\pm Y_0)$, and $(X_1,\pm Y_1)$ for $p=5$, as defined by~\eqref{eqn:X0Y0} and~\eqref{eqn:X1Y1} respectively.
Let $J_p:=\Jac(C_p)/K$ denote the Jacobian of $C_p$.

\subsection{The case \texorpdfstring{$p=5$}{p=5}}

We will show that Corollary~\ref{Cor:GFE-via-rational-points} applies when $p=5$.
\begin{proposition}\label{prop:C5K}
We have
\begin{equation}\label{eqn:C5K}
C_5(K) = \{ (X_0,\pm Y_0), (X_1,\pm Y_1), \infty \}.
\end{equation}
\end{proposition}
With \Magma, we compute that $J_5(K)$ has $2$-Selmer rank equal to $2$ and has trivial $2$-torsion.
Moreover, $[(X_0,Y_0)-\infty], [(X_1,Y_1) - \infty] \in J_5(K)$ provide $2$ independent points of infinite order.
So $\rank J_5(K)=2$ and we have explicit generators for a finite index subgroup.
This puts us in a position to apply Siksek's Chabauty over number fields~\cite{Siksek2013}.

We use the corresponding \Magma\ implementation due to Siksek from~{\it loc.~cit.} (for which we provided minor updates in order to work with recent versions of \Magma).
We find that for the set of primes (in $K$) above $p=47$, the combination of number field Chabauty and the Mordell--Weil sieve are enough to prove that there is only one $K$-rational point in each of the residue disks around $\{ (X_0,\pm Y_0), (X_1,\pm Y_1), \infty \}$, and that all other residue disks do not contain $K$-rational points.
This shows that $C_5(K)$ is as claimed in~\eqref{eqn:C5K}, thereby proving the proposition above.

This finalizes the proof of Theorem~\ref{thm:13-13-5or7} for $p=5$.

\subsection{The case \texorpdfstring{$p=7$}{p=7}}

We will show, assuming GRH, that Corollary~\ref{Cor:GFE-via-rational-points} applies when $p=7$.
\begin{proposition}\label{prop:C7K}
If GRH holds, then
\begin{equation}\label{eqn:C7K}
C_7(K) = \{ (X_0,\pm Y_0), \infty \}.
\end{equation}
\end{proposition}

With \Magma, we compute, \emph{assuming GRH} for the underlying class group computations, that $J_7(K)$ has $2$-Selmer rank equal to $1$ and has trivial $2$-torsion.
Moreover, $[(X_0,Y_0)-\infty] \in J_7(K)$ is a point of infinite order.
So $\rank J_7(K)=1$ and we have an explicit generator for a finite index subgroup.
This puts us in a position to apply \lq standard\rq\ Chabauty (with sieving).
Performing this (e.g. using \Magma) shows that $C_7(K)$ is as claimed in~\eqref{eqn:C7K}, thereby proving the proposition above.

This finalizes the proof of Theorem~\ref{thm:13-13-5or7} for $p=7$.

\subsection{The case \texorpdfstring{$p=11$}{p=11}}

\emph{Assuming GRH} again for the class group computations,
we compute with \Magma\ that $J_{11}(K)$ has $2$-Selmer rank equal to $3$ and has trivial $2$-torsion.
We were able to only find one independent point of infinite order on $J_{11}(K)$, namely $[(X_0,Y_0)-\infty]$.
So $1 \leq \rank J_{11}(K) \leq 3$.
As $C_{11}$ has $K$-rational points, and hence points everywhere locally, we find that if Sha
is finite, then its order is a square (see e.g. \cite[Corollaries 9 and 12]{PoonenStoll99}), and consequently $\rank J_{11}(K) \in \{1, 3\}$.
If by further efforts the rank of $J_{11}(K)$ could be determined, including explicit generators for a finite index subgroup (which we would already have in case the rank would be one), then it is very much expected that \lq standard\rq\ Chabauty (with sieving) applies in practice to determine $C_{11}(K)$.
Alternatively, perhaps Selmer group Chabauty~\cite{Stoll2017} can deal with this case.

\section{Alternative methods}\label{sec:alternative-methods}

To reach our main Theorem~\ref{thm:13-13-5or7}, we have used modular methods and Chabauty methods.
In this section, we discuss the necessity of combining both of these methods, in the sense that we will look into the applicability of Chabauty methods where we have used modular methods, and vice versa, focusing mainly on $p=5$ and $p=7$. We will find that the methods are complementary; without further efforts, we need both of them to reach the main theorem.

Let $p$ be an odd prime and assume $(a,b,c) \in \Z^3$ is a solution to~\eqref{GFE:13-13-p}.

\subsection{Applicability of Chabauty methods}\label{subsec:alternative-Chabauty}

Modular methods are used to take care of two cases.
First, for $13|c$; see Section~\ref{sec:multiFrey} (with $p \geq 5, p \neq 13$).
Next, for $13 \nmid c$ when after descent in~\eqref{eqn:caseI-descent-K} the unit $e=\Norm_{L/K}(\varepsilon_0) \in \calO_K^*$ (up to $p$-th powers, corresponding to the extraneous unit $\varepsilon_0 \in \calO_L^*$); see Section~\ref{subsec:modular-sieve} (with $5 \leq p \leq 47, p \neq 13$).

First, consider the case $13 | c$ (equivalently $13|a+b$).
Similarly as in Section~\ref{sec:HE-curves}, we can reduce to finding rational points on hyperelliptic curves. Let the binary forms $F,G,H \in K[x,y]$ and the constants $d, \pi_{!3} \in K$ be as defined in Section~\ref{sec:HE-curves}.
A classical descent, using~\eqref{eqn:gcd-1or13} and~\eqref{eqn:13-adic_info_Zfactors}, gives that
\[F(a,b) = \pi_{13} e z_1^p \text{ and } 13 G(a,b)=z_2^p\]
for certain $e, z_1, z_2 \in \calO_{K}$ with $e$ a unit (and actually $z_2 \in \Z$).

Write $Y':=H(a,b)/z_2^{2p}$ and $X':=z_1/z_2^4$ as before, where we note that $z_2=0$ corresponds to the trivial solutions with $c=0$.
Specializing~\eqref{eqn:FGH} at $(a,b)$ and dividing by $z_2^{4p}$, we arrive at
\[D'_{p,e}: \quad (1+d)Y'^2=\pi_{13} e X'^p+\frac{d}{13^4}.\]
This defines a hyperelliptic curve $D_{p,e}'$ of genus $(p-1)/2$.
Its equation could be conveniently rescaled again of course.

The case $e=1$ is again of special interest to us, and we write $D'_p:=D'_{p,1}$.
We note that the trivial solutions $(a,b,c)$ with $c=0$ (i.e. $(\pm 1,\mp 1, 0)$) give rise to the point at infinity on $D'_p$.
For $p\not=13,17$ prime with $5\leq p \leq 47$, the basic (non-modular) sieve eliminates all units $e$, except $e=1$ of course; see Case (II) in Section~\ref{subsec:non-modular-sieve}.
We are then left with determining $D'_p(K)$ in order to find the solutions $(a,b,c)$ with $13|c$.

Let us focus on $p \in \{5,7\}$.
Using \Magma, by $2$-Selmer group computations assuming GRH, the rank of the Jacobian of $D'_p/K$ equals $0$, which then easily leads to $D'_p(K)=\{\infty\}$. However, the computations without GRH seem infeasible (we aborted the $p=5$ one after it ran for several days), while modular methods deal with the case $13|c$ unconditionally. Of course, the final Diophantine result, Theorem~\ref{thm:13-13-5or7}, is still conditional on GRH in the $p=7$ case (due to determining $C_7(K)$).

The next case, $13 \nmid c$ and $e=\Norm_{L/K}(\varepsilon_0) \in \calO_K^*$ in~\eqref{eqn:caseI-descent-K}, corresponds to finding $K$-rational points on $C_{p,e}$, as given by~\eqref{eqn:Cpe}.

Let us focus again on $p \in \{5,7\}$.
Using \Magma, by a $2$-Selmer group and a $2$-torsion computation, the rank of the Jacobian of $C_{5,e}/K$ equals $0$ or $1$. If Sha is finite, the rank is $1$, but we did not find a point of infinite order. 
Similarly, but assuming GRH this time, we only find that the rank of the Jacobian of $C_{7,e}/K$ is $\leq 2$ (and not $1$ if Sha is finite).
Without further efforts, we cannot determine $C_{p,e}(K)$ for $p\in \{5,7\}$, while modular methods can deal (unconditionally) with the case $13 \nmid c$ and unit $e$.
With more sieving/searching it may be possible to find a point of infinite order if $p=5$ and 2 independent points of infinite order if $p=7$ (or determine that the rank is $0$ then by other means). Alternatively, perhaps Selmer group Chabauty could be applied.

\begin{remark}
Assuming GRH and finiteness of Sha, we can also compute that both $C_{11,e}/K$ and $D'_{11}/K$ have a rank one Jacobian. But again, we did not find a point of infinite order.
\end{remark}

\subsection{Applicability of modular methods}\label{subsec:alternative-modular-method}

Chabauty methods are used when $13 \nmid c$ (equivalently $13 \nmid a+b$), more specifically, in the case corresponding to the trivial solutions with $ab=0$.
This is done by determining $C_5(K)$ and, assuming GRH, $C_7(K)$; see Section~\ref{sec:Cp-points}.

In view of recent progress~\cite{BCDF2, ChenK} surrounding the Darmon program for the generalized Fermat equation~\cite{Darmon}, it is natural to wonder if we could replace the use of Chabauty methods in our proofs with an application of the multi-Frey modular method using higher dimensional Frey varieties.
Recall that, in Theorem~\ref{thm:1313n}, we used the modular method with Frey curves to deal with the case $13 \mid a+b$ of equation~\eqref{GFE:13-13-p}; in particular, this deals with the trivial solution $\pm (1,-1,0)$.
When $p=5$, the Frey hyperelliptic curve $J = J_5^-(a,b,c)$ attached to Fermat equations of signature $(q,q,5)$ (used with $q=13$ of course) is an excellent candidate for the complementary case, because it becomes singular when evaluated at the trivial solutions $\pm(1,0,1)$ and $\pm(0,1,1)$ (see~\cite[\S 5]{ChenK} for a discriminant formula).
However, a closer look at the proof of~\cite[Theorem 1.2]{ChenK} reveals at least one serious obstruction, namely, when $2 \mid ab$ and $5 \mid ab$ we do not have irreducibility of the mod~$\fq$ representation~$\rhobar_{J,\fq}$.
Unfortunately, our additional assumption $13 \nmid a+b$ does not help and so irreducibility can only be guaranteed conjecturally for large enough~$q$ via~\cite[Conjecture 4.1]{Darmon}.
Additionally, we observe that the Frey hyperelliptic curve for signature $(13,13,p)$ studied in~\cite{BCDF2} is non-singular when evaluated at $\pm(1,0,1)$ and $\pm(0,1,1)$, giving rise to obstructions in the elimination step.
Since this obstructing variety has CM, we only expect to complete the argument conjecturally for large~$p$, assuming the large image conjecture~\cite[Conjecture 4.1]{Darmon}.

\bibliographystyle{plain}
\bibliography{GFE1313n}

\end{document}